\documentclass[runningheads]{llncs}
\usepackage[T1]{fontenc}

\newtheorem{observation}{Observation}
\usepackage{amsfonts, verbatim}
\usepackage{amsmath}
\usepackage{amssymb}
\usepackage{thm-restate}
\usepackage{xcolor}
\usepackage{cite}
\usepackage{xpatch}
\usepackage{apptools}

\usepackage[bookmarks]{hyperref}
 \usepackage{pdfpages}

\makeatletter
\xpatchcmd{\thmt@restatable}
{\csname #2\@xa\endcsname\ifx\@nx#1\@nx\else[{#1}]\fi}
{\ifthmt@thisistheone\csname #2\@xa\endcsname\ifx\@nx#1\@nx\else[{#1}]\fi\else\csname #2\@xa\endcsname\fi}
{}{} 
\makeatother

\DeclareMathOperator{\dimth}{dim_{TH}}
\DeclareMathOperator{\covth}{cov_{TH}}
\DeclareMathOperator{\dimint}{dim_{INT}}
\DeclareMathOperator{\dimsplit}{dim_{SPLIT}}
\DeclareMathOperator{\boxi}{box}
\DeclareMathOperator{\tw}{tw}
\DeclareMathOperator{\pw}{pw}

\title{Bounding threshold dimension: realizing graphic Boolean functions as the AND of majority gates}
\author{Mathew C. Francis\inst{1}
\and Atrayee Majumder\inst{2}
\and Rogers Mathew\inst{3}}
\institute{Indian Statistical Institute, Chennai, India. \email{mathew@isichennai.res.in}\and
Department of Computer Science and Engineering, Indian Institute of Technology, Kharagpur, India. \email{atrayee.majumder@iitkgp.ac.in}\and
Department of Computer Science and Engineering, Indian Institute of Technology, Hyderabad, India. \email{rogers@cse.iith.ac.in}}
\titlerunning{Bounding threshold dimension}
\date{}
\begin{document}
\maketitle
\begin{abstract}
A graph $G$ on $n$ vertices is a \emph{threshold graph} if there exist real numbers $a_1,a_2, \ldots, a_n$ and $b$ such that the zero-one solutions of the linear inequality $\sum \limits_{i=1}^n a_i x_i \leq b$ are the characteristic vectors of the cliques of $G$. Introduced in [Chv{\'a}tal and Hammer, Annals of Discrete Mathematics, 1977], the \emph{threshold dimension} of a graph $G$, denoted by $\dimth(G)$, is the minimum number of threshold graphs whose intersection yields $G$. Given a graph $G$ on $n$ vertices, in line with Chv{\'a}tal and Hammer, $f_G\colon \{0,1\}^n \rightarrow \{0,1\}$ is the Boolean function that has the property that $f_G(x) = 1$ if and only if $x$ is the characteristic vector of a clique in $G$. A Boolean function $f$ for which there exists a graph $G$ such that $f=f_G$ is called a \emph{graphic} Boolean function. It follows that for a graph $G$, $\dimth(G)$ is precisely the minimum number of \emph{majority} gates whose AND (or conjunction) realizes the graphic Boolean function $f_G$. The fact that there exist Boolean functions which can be realized as the AND of only exponentially many majority gates motivates us to  study threshold dimension of graphs. We give tight or nearly tight upper bounds for the threshold dimension of a graph in terms of its treewidth, maximum degree, degeneracy, number of vertices, size of a minimum vertex cover, etc. We also study threshold dimension of random graphs and graphs with high girth. 
\end{abstract}
\section{Introduction}
All the graphs that are mentioned in this paper are finite, simple, and undirected. Given a graph $G=(V,E)$, we shall use $V(G)$ and $E(G)$ to denote the vertex set and edge set of $G$, respectively. For any $v \in V (G)$, we use $N_G(v)$ to denote the
neighborhood of $v$ in $G$, i.e., $N_G(v) = \{u \in V (G)\colon vu \in E(G)\}$. We use $N_G[v]$ to denote $N_G(v) \cup \{v\}$.   For any $S \subseteq V(G)$, 
we shall use $G[S]$ to denote the subgraph induced by the vertex set $S$ in $G$. We use $G-S$ to denote the graph $G[V(G) \setminus S]$. 
A subset of vertices in a graph forms a \emph{clique} if each pair of vertices in this subset has an edge between them; if no pair of vertices have an edge between them, then the subset is called an \emph{independent set}.
\subsubsection{Graphic Boolean functions}
Given a graph $G$ on $n$ vertices, we define the Boolean function $f_G\colon \{0,1\}^n \rightarrow \{0,1\}$ as follows: $\forall x \in \{0,1\}^n$, $f_G(x) = 1$ if and only if $x$ is the characteristic vector of a clique in $G$. A Boolean function $f$ such that there exists a graph $G$ for which $f_G=f$ is called a \emph{graphic} Boolean function. Graphic Boolean functions were defined by Chv\'atal and Hammer~\cite{chvtal1977aggregation} (they defined the Boolean function corresponding to a graph $G$ to be the function whose solutions are exactly the characteristic vectors of the independent sets of $G$; it is easy to see that this is the function $f_{\overline{G}}$ and hence this definition and the one that we gave above for graphic Boolean functions are equivalent). Below, we give a characterization of graphic Boolean functions due to Hammer and Mahadev~\cite{hammermahadev}.

\begin{proposition}[Hammer and Mahadev~\cite{hammermahadev}]
\label{2cnf}
A Boolean function on $n$ variables $x_1, x_2, \ldots , x_n$ is graphic if and only if it can be written in conjunctive normal form where each clause is of the form $(\overline{x_i} \lor \overline{x_j})$, for some distinct $i,j \in [n]$.  
\end{proposition}

\begin{proof}
Given a graph $G$ with vertex set $[n]$, it can be verified that $$f_G = \land_{i,j \in [n]\colon i\neq j,\,ij \notin E(G)}(\overline{x}_i \lor \overline{x}_j).$$ Given an $S \subseteq [n] \times [n]$ and a Boolean function $f = \land_{(i,j) \in S\colon i \neq j}(\overline{x}_i \lor \overline{x}_j)$, consider the graph $G$ on vertex set $[n]$ such that $ij \in E(G)$ if and only if $(\overline{x}_i \lor \overline{x}_j)$ is not a clause in $f$. It can be seen that $f=f_G$.\hfill\qed
\end{proof}
\subsubsection{Majority gates and LTFs}
A \emph{majority gate} is a logic gate that produces an output of 1 if and only if at least half of its input bits are 1. It can be easily seen that an AND or OR gate can be realized using a majority gate by the addition of a suitable number of hardcoded input bits.
A Boolean function $f\colon \{0,1\}^n\rightarrow \{0,1\}$ is called a \emph{Linear Threshold Function} (LTF) if there exists a linear inequality $I\colon \sum \limits_{i=1}^n a_i x_i \leq b$ on variables $x_1,x_2,\ldots,x_n$ such that $\forall x=(x_1,x_2,\ldots,x_n) \in \{0,1\}^n$, $f(x) = 1$ if and only if $x$ satisfies $I$. We say that the linear inequality $I$ ``represents'' $f$.
It is well known that every LTF can be represented by a linear inequality in which the coefficients 
$a_1,a_2,\ldots,a_n,b$ are integers (from here onward, a linear inequality representing an LTF shall be implicitly assumed to have integer coefficients). This implies the well known fact that every LTF can be realized using a majority gate by wire duplication.
Conversely, it is easy to see that any Boolean function that can be realized using a majority gate
is an LTF.

\subsubsection{Threshold graphs}
A graph $G$ on $n$ vertices is a \emph{threshold graph} if there exist real numbers $a_1,a_2, \ldots, a_n$ and $b$ such that the zero-one solutions of the linear inequality $\sum \limits_{i=1}^n a_i x_i \leq b$ are the characteristic vectors of the cliques of $G$. This implies that $G$ is a threshold graph if and only if $f_G$ is an LTF. Since LTFs are exactly the Boolean functions that can be realized using a majority gate, we can equivalently say that a graph $G$ is a threshold graph if and only if $f_G$ can be realized using a majority gate. 
Chv\'atal and Hammer~\cite{chvtal1977aggregation} showed that threshold graphs are exactly the graphs that contain no induced subgraph isomorphic to $2K_2$, $P_4$ or $C_4$ (the graph with four vertices and two disjoint edges, the path on four vertices and the cycle on four vertices respectively). Thus, the complement of a threshold graph is also a threshold graph, implying that one can replace `cliques' with `independent sets' in the  definition of a threshold graph.  
The complete graph on $n$ vertices is a threshold graph with the corresponding linear inequality being $\sum\limits_{i=1}^n x_i\leq n$. Similarly, the star graph $K_{1,n-1}$ is a threshold graph, as shown by the linear inequality $x_1+\sum\limits_{i=2}^n (n-1)x_i\leq n$.
For a graph $G$, the characteristic vectors of the subsets of $V(G)$ correspond to the corners of the $n$-dimensional hypercube. Thus, a graph $G$ is threshold if and only if there is a hyperplane in $\mathbb{R}^n$ that separates the corners of the $n$-dimensional hypercube that correspond to the cliques of $G$ from the other corners of the hypercube. 
Threshold graphs, which find applications in integer programming and set packing problems, were introduced by Chv{\'a}tal and Hammer~\cite{chvtal1977aggregation}. 
Refer to the book \cite{golumbic2004algorithmic} by Golumbic to know more about the different properties of threshold graphs. A more comprehensive study of threshold graphs can be found in the book~\cite{mahadev1995threshold} by Mahadev and Peled. 

The following equivalent characterization of threshold graphs (Corollary~1B in~\cite{chvtal1977aggregation}) will be useful for us.
\begin{proposition}[Chv{\'a}tal and Hammer~\cite{chvtal1977aggregation}]
\label{defn_neighborhood}
$G$ is a threshold graph if and only if there is a partition of $V(G)$ into an independent set $A$ and a clique $B$, and an ordering $u_1, u_2, \ldots, u_k$ of $A$ such that
$N_G(u_k) \subseteq N_G(u_{k-1}) \subseteq \cdots \subseteq N_G(u_1)$.
\end{proposition}

\subsubsection{Threshold dimension}
If $G_1, G_2,\ldots,G_k$ are graphs on the same vertex set as $G$ such that $E(G) = E(G_1) \cap E(G_2) \cap \cdots \cap E(G_k)$, then we say that $G=G_1\cap G_2\cap\cdots\cap G_k$. In a similar way, if $E(G) = E(G_1) \cup E(G_2) \cup \cdots \cup E(G_k)$, then we say that $G=G_1\cup G_2\cup\cdots\cup G_k$. 
Given a class $\mathcal{A}$ of graphs, Kratochv\'il and Tuza~\cite{kratotuza} defined the \emph{$\mathcal{A}$-dimension} of a graph $G$, denoted as $\dim_{\mathcal{A}}(G)$, to be the minimum integer $k$ such that there exist $k$ graphs in $\mathcal{A}$ whose intersection is $G$. 
Let $\mathrm{TH}$ denote the class of threshold graphs. Chacko and Francis~\cite{chacko2020representing} studied the parameter $\dimth(G)$ of a graph $G$, which in the language of~\cite{kratotuza}, can be called the \emph{threshold dimension} of $G$. 
\begin{definition}[Threshold dimension]
The threshold dimension of a graph $G$, denoted by $\dimth(G)$, is the smallest integer $k$ for which there exist threshold graphs $G_1,G_2,\ldots,G_k$ such that $G=G_1\cap G_2\cap\cdots\cap G_k$.
\end{definition}

Let $f\colon \{0,1\}^n \rightarrow \{0,1\}$ be a Boolean function. Let $\gamma(f)$ denote the minimum number of LTFs whose AND (or conjunction) realizes $f$, or equivalently, the minimum number of majority gates in a depth-2 circuit realizing $f$ whose first layer consists of only majority gates and second layer consists of a single output AND gate. 
 Chv\'atal and Hammer proved the following theorem connecting the parameters $\gamma(f_G)$ and $\dimth(G)$ for a graph $G$.

\begin{theorem}[Chv{\'a}tal and Hammer~\cite{chvtal1977aggregation}]
\label{thm:chvham}
\label{thm_alpha_dimth}
For a graph $G$, $\gamma(f_G) = \dimth(G)$.
\end{theorem}

\begin{proof}
	Let $[n]$ be the vertex set of $G$. For any distinct $i,j \in [n]$, we use $v^{i,j}$ to denote the $n$-bit $0$-$1$ vector which has a $1$ only at the $i$th and $j$th bit positions. 
	
	Suppose $\dimth(G) = k$. Then there exist $k$ threshold graphs, namely $G_1,$  $G_2, \ldots , G_k$ such that $G = G_1 \cap G_2 \cap \cdots \cap G_k$. We know that corresponding to each threshold graph $G_{\ell}$, for $1\leq\ell\leq k$, there is an LTF $f_{\ell}$ such that $f_{G_{\ell}} = f_{\ell}$. It is not difficult to see that $f_G = \land_{\ell = 1}^k f_{\ell}$ and therefore $\gamma(f_G) \leq \dimth(G)$. 
	
	To prove that $\dimth(G) \leq \gamma(f_G)$, assume $\gamma(f_G) = k$. Then there exist $k$ LTFs $f_1, \ldots , f_k$ such that $f_G = \land_{\ell = 1}^k f_{\ell}$. For each $\ell \in [k]$, we construct a threshold graph $G_{\ell}$ as described below. For every distinct $i,j \in [n]$, we let $ij \in E(G_{\ell})$ if and only if $f_{\ell}(v^{i,j}) = 1$. It can be seen that $G_{\ell}$ cannot contain a $2K_2$, $P_4$ or $C_4$ as an induced subgraph, and hence is a threshold graph. It remains to show that $G = \cap_{\ell = 1}^k G_{\ell}$. Consider any distinct $i,j\in [n]$. Suppose that $ij \in E(G)$. Then $f_G(v^{i,j}) = f_1(v^{i,j}) = \cdots = f_k(v^{i,j}) = 1$. Our construction of $G_{\ell}$ ensures that $ij \in E(G_{\ell})$, for every $\ell \in [k]$. Next suppose that $ij \notin E(G)$. Then there exist some $\ell\in [k]$ such that $f_{\ell}(v^{i,j}) = 0$. Then, from our construction, $ij \notin E(G_{\ell})$. We have thus shown that $G = G_1 \cap \cdots \cap G_k$, and therefore $\dimth(G)\leq\gamma(f_G)$.
	\hfill\qed
\end{proof}

For any Boolean function $f$ on $n$ variables, $\gamma(f)\leq 2^n$ (since any Boolean function on $n$ variables can be realized using a depth-2 circuit in which the first layer contains at most $2^n$ OR gates and the second layer contains an AND gate --- which is just another way of saying that $f$ can be written in conjunctive normal form), and there are families of Boolean functions $\{f^1,f^2,\ldots\}$, where $f^i$ is a Boolean function on $i$ variables, for which $\gamma(f^n)$ is exponential in $n$~\cite{mahajan}. For a Boolean function $f$ on $n$ variables that can be expressed as a 2-CNF formula, the number of clauses in it is at most ${2n\choose 2}$, which means that $f$ can be realized using a depth-2 circuit containing at most $2n\choose 2$ majority gates. If further, $f$ is a graphic Boolean function, then the number of clauses when written in 2-CNF form is at most $n\choose 2$ (by Proposition~\ref{2cnf}), implying that $f$ can be realized using a depth-2 circuit containing at most $n\choose 2$ majority gates.
As for any graph $G$, we have $\gamma(f_G)=\dimth(G)\leq n$ (this can be seen as follows: for every vertex $u\in V(G)$, we define the graph $G_u$ on vertex set $V(G)$ and having edge set $\{xy\colon$  
$x,y\in V(G)\setminus\{u\}$ and $x \neq y\}\cup\{uv\colon v\in N_G(u)\}$; then $G=\bigcap_{u\in V(G)} G_u$ and each $G_u$ is a threshold graph), every graphic Boolean function on $n$ variables can be realized using a depth-2 circuit whose first layer contains at most $n$ majority gates. This can be improved further by deriving better upper bounds for threshold dimension (see for example, Corollary~\ref{cor:uppboundn}). Further, when the graphs corresponding to the graphic Boolean functions have some nice properties, we can show even better bounds on the number of majority gates required in a depth-2 circuit realizing the function.


Note that Chv\'atal and Hammer~\cite{chvtal1977aggregation} use the term ``threshold dimension'' of a graph $G$ with a slightly different meaning: they define it to be the minimum integer $k$ for which there exist threshold graphs $G_1,G_2,\ldots,G_k$ such that $G=G_1\cup G_2\cup\cdots\cup G_k$. We call this the \emph{threshold cover number} of $G$ and denote it by $\covth(G)$. Since the complement of a threshold graph is also a threshold graph, we have the following.
\begin{observation}
\label{obv_threshcov_threshdim}
For a graph $G$, $\covth(G) = \dimth(\overline{G})$. 
\end{observation}
For a graph $G$, let $\alpha(G)$, $\omega(G)$, and $\chi(G)$ denote the size of a maximum independent set, the size of a maximum clique, and the chromatic number of $G$, respectively. 
It was shown in~\cite{chvtal1977aggregation} that for every graph $G$ on $n$ vertices, $\covth(G) \leq n - \alpha(G)$. 
In the same paper, the authors  also showed that for every positive $\epsilon$, there is a graph $G$ on $n$ vertices such that $\covth(G) > (1-\epsilon)n$. 
Yannakakis~\cite{yannakakis1982complexity} showed that it is NP-complete to recognize graphs having threshold cover number at most $k$, for all fixed $k \geq 3$. Raschle and Simon~\cite{raschle1995recognition} showed that there is a polynomial time algorithm that recognizes graphs having threshold cover number at most $2$. 
Combining Observation \ref{obv_threshcov_threshdim} with the results due in~\cite{chvtal1977aggregation, yannakakis1982complexity, raschle1995recognition} mentioned above directly yields the following.
\begin{corollary}
\label{thm_follows_from_thr_cov}
~
\begin{enumerate}
\renewcommand{\theenumi}{\alph{enumi}}
\renewcommand{\labelenumi}{(\theenumi)}
\item For a graph $G$ on $n$ vertices, $\dimth(G) \leq n - \omega(G)$, where $\omega(G)$ denotes the size of a largest clique in $G$.
\item For every positive $\epsilon$, there is a graph $G$ on $n$ vertices such that $\dimth(G) > (1-\epsilon)n$.
\item It is NP-complete to recognize graphs having threshold dimension at most $k$, for all fixed $k \geq 3$.
\item There is a polynomial time algorithm that recognizes graphs having threshold dimension at most $2$.
\end{enumerate}
\end{corollary}

We now give a lower bound on the threshold dimension of a graph.

\begin{proposition}
\label{prop:lowboundchi}
For a graph $G$, $\dimth(G)\geq\min\{\chi(G-C)\colon C$ is a clique of $G\}$.
\end{proposition}

\begin{proof}
Suppose that $G$ is a graph and $G_1,G_2,\ldots,G_k$ are threshold graphs such that $G=G_1\cap G_2\cap\cdots\cap G_k$. By Proposition~\ref{defn_neighborhood}, we have that for each $i\in [k]$, there is a partition of $V(G_i)$ into an independent set $A_i$ and a clique $B_i$.
It is not difficult to see that $B=B_1\cap B_2\cap\cdots\cap B_k$ is a clique of $G$, and each $A_i$, for $i\in [k]$, is an independent set of $G$. Since $V(G)\setminus B=A_1\cup A_2\cup\cdots\cup A_k$, we have that $V(G)\setminus B$ is the union of $k$ independent sets of $G$. This implies that $V(G)\setminus B$ can also be partitioned into $k$ independent sets of $G$, and therefore $\chi(G-B)\leq k$. Thus there always exists a clique $B$ in $G$ such that $k\geq\chi(G-B)$. This completes the proof.\hfill\qed
\end{proof}
Note that the above proposition actually gives a lower bound on $\dimsplit(G)$, where $\mathrm{SPLIT}$ is the class of ``split graphs'' --- the graphs whose vertex set can be partitioned into an independent set and a clique --- of which the class of threshold graphs is a subclass.

A graph is an \emph{interval graph} if there is a mapping from the set of vertices of the graph to the set of closed intervals on the real line such that two vertices in the graph are adjacent to each other if and only if the intervals they are mapped to have a non-empty intersection. Let $\mathrm{INT}$ denote the class of interval graphs. The parameter $\dimint(G)$ is more commonly known as the \emph{boxicity} of the graph $G$ and denoted as $\boxi(G)$. 
It is known that threshold graphs form a subclass of the class of interval graphs.  
This implies the following.
\begin{observation}  
\label{obv_thdim_box}
For a graph $G$, $\boxi(G) \leq \dimth(G)$. 
\end{observation}
The graph parameter `boxicity' was introduced by Roberts \cite{Roberts} in  1969 and, since then, it has been extensively studied (see \cite{CHANDRAN2007733, AdiSunRog, tech-rep, boxMinor, Kratochvil, DiptAdigaHardness}). 
We will see how Observation \ref{obv_thdim_box} helps us get tight examples  to various bounds we prove for threshold dimension in this paper. Chacko and Francis~\cite{chacko2020representing} gave the following upper bound for the threshold dimension of a graph $G$ in terms of its boxicity and chromatic number.
\begin{theorem}[Theorem 19 in \cite{chacko2020representing}]
\label{thm_mathew_chacko_box_threshdim}
For a graph $G$, $\dimth(G) \leq \boxi(G)\cdot \chi(G)$.
\end{theorem}
We note here that the above upper bound is tight, as shown by the following observation, which also shows that the threshold dimension of a graph cannot be bounded by any function of its boxicity.

\begin{proposition}
\label{obv_int_dimth}
There is an interval graph $G$ for which $\dimth(G)=\chi(G)=|V(G)|/2$.
\end{proposition}

\begin{proof}
Consider the graph $2K_n$. This graph is clearly an interval graph, and removing any clique from this graph results in a graph that contains a clique of $n$ vertices. Thus by Proposition~\ref{prop:lowboundchi}, we have that $\dimth(2K_n)\geq n=\chi(2K_n)$. Theorem~\ref{thm_mathew_chacko_box_threshdim} implies $\dimth(2K_n)\leq n$.\hfill\qed
\end{proof}

In this paper, we prove tighter upper bounds for the threshold dimension of a graph that cannot be obtained from Theorem \ref{thm_mathew_chacko_box_threshdim} by plugging in known upper bounds for boxicity.
\subsection{Our results}
Let $G$ be a graph with $n$ vertices. Let $\Delta$ denote the maximum degree of a vertex in $G$ and let $\tw(G)$ denote the treewidth of $G$. Let $\alpha(G)$ and $\omega(G)$ denote the sizes of a maximum independent set and a maximum clique, respectively, in $G$. We prove the following results. 
\begin{enumerate}
\item Chandran and Sivadasan~\cite{CHANDRAN2007733} showed that for any graph $G$, $\boxi(G)\leq\tw(G)+2$. Chacko and Francis~\cite{chacko2020representing} note that for any graph $G$, $\dimth(G)\leq (\tw(G)+1)(\tw(G)+2)$ and ask if the threshold dimension of every graph can be bounded by a linear function of its treewidth. In Section \ref{sec_treewidth}, we answer this question in the affirmative by showing that $\dimth(G) \leq 2(\tw(G) +1)$. We show that this bound is tight up to a multiplicative factor of 2. Co-comparability graphs, AT-free graphs, and chordal graphs are known to have $O(\Delta)$ upper bounds on their treewidth. 
We thus get an $O(\Delta)$ upper bound to the threshold dimension of such graphs.

\item  Let $\dimth(\Delta) := \max \{\dimth(G)\colon G$ is a graph having maximum degree $\Delta \}$. In Section \ref{sec_maxdegree}, we show that $\dimth(\Delta) = O(\Delta \ln^{2+o(1)}\Delta)$. It was shown by Erd\H{o}s, Kierstead, and Trotter in \cite{erdos1991dimension} that there exist graphs $G$
having boxicity $\Omega(\Delta \ln\Delta)$. Using Observation \ref{obv_thdim_box}, we get $\dimth(\Delta) = \Omega(\Delta \ln \Delta)$. Bridging the gap between the upper and lower bounds for $\dimth(\Delta)$ would be interesting. Since, by Theorem~\ref{thm_alpha_dimth}, $\dimth(G) = \gamma(f_G)$, it may be worthwhile to see if techniques from complexity theory could be used to bridge this gap. 
\item Let $G$ be $k$-degenerate. We show in Section \ref{sec_degeneracy} that $\dimth(G) \leq 10k\ln n$. It was shown in Section 3.1 in \cite{AdiSunRog} that there exist $k$-degenerate graphs on $n$ vertices with boxicity in $\Omega(k\ln n)$. Together with Observation \ref{obv_thdim_box}, this implies that the upper bound for $\dimth(G)$ we prove in Section \ref{sec_degeneracy} is tight up to constants. This bound gives some interesting corollaries.
\begin{enumerate}
\item Let $G \in \mathcal{G}(n,m)$, where $m \geq  n/2$. Then, asymptotically almost surely $\dimth(G) \in O(d_{av}\log n)$, where $d_{av} = \frac{2m}{n}$ denotes the average degree of $G$. 
\item If $G$ has a girth greater than $g+1$, then $\dimth(G) = O(n^{\frac{1}{\lfloor g/2 \rfloor }}\ln n)$. 
\end{enumerate}  
\item In Section \ref{sec:thdim_MVC}, we show that the threshold dimension of any graph is upper bounded by its minimum vertex cover number, which implies that for any graph $G$, $\dimth(G) \leq n-\max\{\alpha(G), \omega(G)\}$. We show that this bound is tight. As a corollary we show that if $n$ is sufficiently large, then $\dimth(G) \leq n - 0.72\ln n$.   
\end{enumerate}
\subsection{Preliminaries}
\begin{definition}
\label{def:thsupgraph}
Given a graph $G$, an independent set $A = \{u_1, u_2,\ldots , u_t\}$ in $G$, and a total ordering $\sigma\colon u_1, u_2, \ldots , u_t$ of the vertices of $A$, we define the threshold supergraph $\tau(G,A,\sigma)$ of $G$ as below. Let $B = V(G) \setminus A$ and for $v\in B$, let $s(v)=\max\{i\colon u_i\in N_G(v)\}$ if $N(v)\cap B\neq\emptyset$ and $s(v)=0$ otherwise. In $\tau(G,A,\sigma)$, the vertices of $A$ form an independent set and those of $B$ form a clique and each vertex $v \in B$ is adjacent to exactly the vertices $u_1, u_2, \ldots, u_{s(v)}$. Formally, $$V(\tau(G,A,\sigma))=V(G)$$ $$E(\tau(G,A,\sigma))=E(G)\cup\{xy\colon x,y\in B \emph{ and } x \neq y\}\cup\bigcup_{v\in B}\{vu_1,vu_2,\ldots,vu_{s(v)}\}$$
\end{definition}
The following proposition follows directly from the above definition and Proposition~\ref{defn_neighborhood}.
\begin{proposition}
\label{prop:thsupgraph}
Given a graph $G$, an independent set $A$ of $G$, and an ordering $\sigma$ of $A$, the graph $\tau(G,A,\sigma)$ is a threshold graph and $G$ is its subgraph.
\end{proposition}
\section{Threshold dimension and treewidth}
\label{sec_treewidth}
In this section, we show that, for a graph $G$, $\dimth(G) \leq 2(\tw(G) + 1)$, where $\tw(G)$ denotes the treewidth of $G$. 
We set up some notations and discuss some necessary existing results before going into the proof of the main result. 
\subsection{Definitions, notations, and known results}
The notion of treewidth was first introduced by Robertson and Seymour in~\cite{ROBERTSON1986309}.
\begin{definition}[Tree decomposition]
A \emph{tree decomposition} of a graph $G=(V,E)$ is a pair 
	$(T,\{X_i \colon  i \in V(T)\})$ where $T$ is a tree and for each $i\in V(T)$, $X_i$ is a subset of $V(G)$ (sometimes called a bag), such that the following conditions are satisfied:
	\begin {itemize}
	\item  $\bigcup_{i \in V(T)} X_i = V(G)$.
	\item  $ \forall uv \in E(G), \exists i \in V(T)$, such that $u,v\in X_i$.
	\item  $ \forall i,j,k \in V(T)$: if $j$ is on the path in $T$
	       from $i$ to $k$, then $X_i \cap X_k \subseteq  X_j$.
	\end {itemize}
\end{definition}
The \emph{width} of a tree-decomposition $(T,\{X_i \colon  i \in V(T)\})$ is $\max_{i \in V(T)} |X_i| - 1$.
\begin{definition}[Treewidth]
The \emph{treewidth} of a graph $G$, denoted by $\tw(G)$, is the minimum width over all possible tree decompositions of $G$.
\end{definition}

A tree decomposition $(T, \{X_i \colon  i \in V(T)\})$ of a graph $G$ is said to be a \emph{path decomposition} of $G$ if $T$ is a path.
The \emph{pathwidth} of G, denoted by $\pw(G)$, is defined as the minimum width over all possible path decompositions of $G$. The following result by Chacko and Francis connects threshold dimension of a graph with its pathwidth.
\begin{theorem}[Theorem 7 in \cite{chacko2020representing}]
For every graph $G$, $\dimth(G)\leq \pw(G) + 1$.
\end{theorem}
Since path decompositions are special cases of tree decompositions, it can be seen that $\tw(G) \leq \pw(G)$. Korach and Solel showed that $\pw(G) = O(\log{n} \cdot \tw(G))$, where $n = |V(G)|$ (Theorem 6 in \cite{KORACH199397}). We thus have $\dimth(G) = O(\log{n} \cdot \tw(G))$. Chacko and Francis note that for any graph $G$, $\dimth(G)\leq (\tw(G)+1)(\tw(G)+2)$ and ask if there is a linear bound on the threshold dimension of a graph in terms of its treewidth. We give an affirmative answer to this question.                                                                                      

Given an ordering $\sigma$ of the vertices of a graph $G$ and $u,v\in V(G)$, we denote by $u<_{\sigma} v$ the fact that $u$ appears before $v$ in the ordering. 

Let $T$ be a rooted tree. For any $u, v \in V(T)$, $u$ is an \emph{ancestor} of $v$, and $v$ a \emph{descendant} of $u$, if $u$ lies on the path from $v$ to the root of $T$. It follows from this definition that every vertex of $T$ is both an ancestor and descendant of itself.
For a rooted tree $T$, a \emph{preorder traversal} of $T$ is an ordering of $V(T)$ in the order in which a depth-first search algorithm starting from the root may visits the vertices of $T$. The following is not difficult to see.
\begin{proposition}\label{prop:preorder}
If $\pi$ is a preorder traversal of a rooted tree $T$, then:
\begin{enumerate}
\vspace{-0.05in}
\itemsep 0.01in
\renewcommand{\theenumi}{(\roman{enumi})}
\renewcommand{\labelenumi}{\theenumi}
\item\label{preorderprop1} for $u,v\in V(T)$ such that $v$ is a descendant of $u$, we have $u<_\pi v$, and
\item\label{preorderprop2} for $u,v,w\in V(T)$ such that $u<_\pi v<_\pi w$, if $w$ is a descendant of $u$, then $v$ is also a descendant of $u$.
\end{enumerate}
\end{proposition}
Let $G$ be a graph and $\mathcal{T}=(T,\{X_i\colon i\in V(T)\})$ be a tree decomposition of $G$ having width $k$. We choose an arbitrary vertex $r$ to be the root of $T$ and henceforth consider $T$ to be a rooted tree. Then a function $b\colon V(G) \rightarrow V(T)$ is defined as follows: for a vertex $v \in V(G)$, $b(v)$ is the bag containing $v$ in the tree decomposition that is closest to $r$. Formally, $b(v)$ is the vertex of $T$ such that $v\in X_{b(v)}$ and $v\notin X_i$ for any $i\in V(T)$ that is an ancestor of $b(v)$.
\begin{lemma}[Lemma 10 in~\cite{CHANDRAN2007733}]
\label{lem:edgecomparable}
If $uv \in E(G)$, then $b(u)$ is either an ancestor or descendant of $b(v)$ in $T$.
\end{lemma}
\begin {lemma}[Lemma 8 in \cite{CHANDRAN2007733}] 
\label{lem:theta}
There exists a function $\theta \colon  V(G) \rightarrow \{0, 1,\ldots,k\}$, such that for any $i \in V(T)$ and for any two distinct nodes  $u, v \in X_i$, $\theta(u) \neq\theta(v)$.
\end{lemma}

\noindent{\textit{Remark.}} The function $\theta$ is a proper vertex colouring of the chordal graph $G'$ that one obtains from $G$ by adding edges between every pair of vertices that appear together in some bag of the tree decomposition. Clearly, $\mathcal{T}$ is a tree decomposition of $G'$ as well. From the fact that every clique in $G'$ has to be contained in some bag of $\mathcal{T}$, and the fact that chordal graphs are perfect, it follows that $\theta$ needs to use only $\max\{|X_i|\colon i\in V(T)\}$ different colours.

The following lemmas from~\cite{CHANDRAN2007733} describe some properties of the functions $\theta$ and $b$ that we will use later. These are direct corollaries of the definition of $\theta$ and that of tree decompositions.
\begin{lemma}[Lemma 9 in \cite{CHANDRAN2007733}]
\label{lem:thetacoloring}
If $uv \in E(G)$ then $\theta(u) \neq \theta(v)$.
\end{lemma}
\begin{lemma}[Lemma 11 in \cite{CHANDRAN2007733}]
\label{lem:middle}
Let $uv \in E(G)$ and let $b(u)$ be an ancestor of $b(v)$. For any vertex $w \in V(G) \setminus \{u\}$, $\theta(w) \neq \theta(u)$ if $b(w)$ is in the path from $b(v)$ to $b(u)$ in $T$.
\end{lemma}
Let $\pi$ be a preorder traversal of $T$. Let $\sigma$ be an ordering of $V(G)$ such that for any two vertices $u,v\in V(G)$, $u<_\sigma v$ in $\sigma$ if $b(u)<_\pi b(v)$. (In $\sigma$, we let the ordering between two vertices $u,v\in V(G)$ such that $b(u)=b(v)$ to be arbitrary. Thus, if $u<_\sigma v$, then $b(u)\leq_\pi b(v)$.) Let $\sigma^{-1}$ denote the ordering of $V(G)$ obtained by reversing the ordering $\sigma$. Given a set $A\subseteq V(G)$, we denote by $\sigma|_A$ the ordering of vertices of $A$ in the order in which they appear in $\sigma$.
\subsection{Proof of the main result}
For $i\in\{0,1,\ldots,k\}$, we define $C_i=\{v\in V(G)\colon\theta(v)=i\}$. From Lemma~\ref{lem:thetacoloring}, we know that $\theta$ is a proper colouring of $G$, which implies that $C_i$ is an independent set of $G$.
For each class $C_i$, where $0 \leq i \leq k$, we define two graphs $G_i^{1}=\tau(G,C_i,\sigma|_{C_i})$ and $G_i^{2}=\tau(G,C_i,\sigma^{-1}|_{C_i})$. 

\begin{lemma}
\label{lem:main}
Let $u,v$ be distinct vertices in $G$. Then there do not exist $x_u,y_u\in N_G(u)$ and $x_v,y_v\in N_G(v)$ such that $x_u<_\sigma v<_\sigma y_u$, $x_v<_\sigma u<_\sigma y_v$, $\theta(u)=\theta(x_v)$, and $\theta(v)=\theta(x_u)$.
\end{lemma}
\begin{proof}
	Clearly, we have either $u<_\sigma v$ or $v<_\sigma u$. Let us assume without loss of generality that $u<_\sigma v$. Then we have $u<_\sigma v<_\sigma y_u$, which implies that $b(u)\leq_\pi b(v)\leq_\pi b(y_u)$. Since $uy_u\in E(G)$, we have from Lemma~\ref{lem:edgecomparable} that $b(u)$ is either an ancestor or descendant of $b(y_u)$. As $\pi$ is a preorder traversal of $T$, Proposition~\ref{prop:preorder}\ref{preorderprop1} implies that $b(u)$ is an ancestor of $b(y_u)$ in $T$. As $b(u)\leq_\pi b(v)\leq_\pi b(y_u)$, it now follows from Proposition~\ref{prop:preorder}\ref{preorderprop2} that $b(v)$ is a descendant of $b(u)$. Similarly, $x_v<_\sigma u<_\sigma v$ implies that $b(x_v)\leq_\pi b(u)\leq_\pi b(v)$, and $vx_v\in E(G)$ then implies by Lemma~\ref{lem:edgecomparable}, Proposition~\ref{prop:preorder}\ref{preorderprop1} and~\ref{preorderprop2} that $b(u)$ is a descendant of $b(x_v)$. Now applying Lemma~\ref{lem:middle} to $x_v$, $u$ and $v$, we have that $\theta(x_v)\neq\theta(u)$, which is a contradiction.
	\hfill\qed
\end{proof}

\begin{lemma}\label{lem:final}
$G = \bigcap \limits_{0 \leq i \leq k} (G_i^1\cap G_i^2)$
\end{lemma}
\begin{proof}
	Consider any two distinct vertices $u$ and $v$ of $G$. Since $G_i^1$ and $G_i^2$, for $1\leq i\leq k$, are both supergraphs of $G$ by definition, we have that if $uv\in E(G)$, then $uv$ is an edge of both $G_i^1$ and $G_i^2$. So in order to prove the lemma, we only need to prove that whenever $uv\notin E(G)$, there exists $i\in\{0,1,\ldots,k\}$ and $j\in\{1,2\}$ such that $uv\notin E(G_i^j)$.
	
	Suppose $\theta(u) = \theta(v) = i$. Since the class $C_i$ is an independent set in $G_{i}^1$ and $G_{i}^2$, $uv$ is an edge in neither $G_i^1$ nor $G_i^2$, and we are done. So let us assume that $\theta(u) \neq \theta(v)$. Let $\theta(u)=i$ and $\theta(v)=j$. We claim that $uv$ is not an edge in one of the graphs $G_i^1$, $G_i^2$, $G_j^1$, or $G_j^2$. Suppose for the sake of contradiction that $uv\in E(G_i^1)\cap E(G_i^2)\cap E(G_j^1)\cap E(G_j^2)$. Then $uv$ is an edge in each of the graphs $\tau(G,C_i,\sigma|_{C_i})$, $\tau(G,C_i,\sigma^{-1}|_{C_i})$, $\tau(G,C_j,\sigma|_{C_j})$, $\tau(G,C_j,\sigma^{-1}|_{C_j})$. Since $uv\in E(\tau(G,C_i,\sigma|_{C_i}))$, by Definition~\ref{def:thsupgraph}, we have that there exists $y_v\in C_i\cap N_G(v)$ such that $u<_\sigma y_v$. Further, since $uv\in E(\tau(G,C_i,\sigma^{-1}|_{C_i}))$, there exists $x_v\in C_i\cap N_G(v)$ such that $u<_{\sigma^{-1}} x_v$, or in other words, $x_v<_\sigma u$. As $uv\in E(\tau(G,C_j,\sigma|_{C_j}))$ and $uv\in E(\tau(G,C_j,\sigma^{-1}|_{C_j}))$, we can similarly conclude that there exist $x_u,y_u\in C_j\cap N_G(u)$ such that $x_u<_\sigma v<_\sigma y_u$. Since $\theta(x_u)=\theta(v)=j$ and $\theta(x_v)=\theta(u)=i$, we now have a contradiction to Lemma~\ref{lem:main}.
	\hfill\qed
\end{proof}
From Proposition~\ref{prop:thsupgraph} and Definition~\ref{def:thsupgraph}, it follows that $G_i^1$ and $G_i^2$ are both threshold graphs for each $i\in\{0,1,2,\ldots,k\}$. Thus by Lemma~\ref{lem:final}, we get that $\dimth(G)\leq 2(k+1)$, which leads to the following theorem.
\begin{theorem}
\label{thm_thresh_dim_treewidth}
For any graph $G$, $\dimth(G) \leq 2(\tw(G)+1)$.
\end{theorem}
\subsubsection{Tightness of the bound}
Note that from Proposition~\ref{obv_int_dimth}, we know that the graph $2K_n$ has threshold dimension $n$ and it is easy to see that the treewidth of this graph is $n-1$. Thus the upper bound on threshold dimension given by Theorem~\ref{thm_thresh_dim_treewidth} is tight up to a multiplicative factor of 2. We give below another example that shows the same tightness result. 

\begin{example}
	\label{example:treewidth_MVC}
	Let $A = \{a_1, \ldots , a_n\}$ and $B = \{b_1, \ldots , b_n\}$. Let $G$ be a graph defined as $V(G) = A \cup B$ and $E(G) = \{a_ia_j\colon 1 \leq i<j\leq n\} \cup \{a_ib_i\colon 1 \leq i \leq n\}$. Let $H$ be the complement of the graph $G$. 
\end{example}
We claim that $\dimth(H) = n$. To show that $\dimth(H) \leq n$, it is easy to see that the edges of $G$ can be covered using $n$ threshold graphs (for each $\ell\in\{1,2,\ldots,n\}$, we can construct a threshold graph having vertex set $V(G)$ and edge set $\{a_ia_j\colon 1\leq i<j\leq n\}\cup\{a_\ell b_\ell\}$; the union of these graphs is $G$). In order to prove that $\dimth(H) \geq n$, assume there is a possibility of representing $H$ as the intersection of less than $n$ threshold graphs. Then there must exist a threshold graph where $a_i$ is non-adjacent to $b_i$ and $a_j$ is non-adjacent to $b_j$, for some $i, j \in [n]$, $i \neq j$. This implies the existence of  an induced $P_4$ (the path $a_ib_jb_ia_j$) in this threshold graph, which is a contradiction. 

Next, we show that $\tw(H) = n-1$. Since $H$ contains a clique of size $n$, $\tw(H) \geq n-1$. Let $X_0 = \{b_1, \ldots , b_n\}$, $X_i = \{a_i\} \cup (B\setminus \{b_i\})$, for all $i \in [n]$. Let $T$ be the tree having vertex set $\{0, 1, \ldots, n\}$ in which the vertex 0 has degree $n$ and all other vertices have degree 1. Observe that the pair $(T,\{X_i\}_{i\in\{0,1,\ldots,n\}})$ is a tree decomposition of $H$ having width $n-1$. Thus, $\tw(H)\leq n-1$. Hence, this example also demonstrates that the bound in Theorem~\ref{thm_thresh_dim_treewidth} is tight up to a multiplicative factor of 2.

\section{Threshold dimension and maximum degree}
\label{sec_maxdegree}
Let $\dimth(\Delta) := \max \{\dimth(G)\colon G$ is a graph having maximum degree $\Delta \}$. In this section, we show that $\dimth(\Delta) = O(\Delta \ln^{2+o(1)}\Delta)$. 

\subsection{Definitions, notations, and auxiliary results}
Given a graph $G$ and an $S \subseteq V(G)$, recall that 
we use $G[S]$ to denote the subgraph induced by the vertex set $S$ in $G$.
For any disjoint pair of sets $S,T \subseteq V(G)$, we use $G[S,T]$ to denote the bipartite subgraph of $G$ where $V(G[S,T]) = S \cup T$ and $E(G[S,T]) = \{uv\colon  u \in S,~v \in T,~ uv \in E(G)\}$. 
Let $G^{*}[S,T]$ denote the graph constructed from $G[S,T]$ by making $T$ a clique. That is, $V(G^{*}[S,T]) = S \cup T$ and $E(G^{*}[S,T]) = E(G[S,T]) \cup \{uv\colon u,v \in T\}$.

We state below the definition of a $k$-suitable family of permutations that was introduced by Dushnik in \cite{dushnik1950concerning}.
\begin{definition}[$k$-suitable family of permutations]
	\label{def:k_suitable_family}
	A family of permutations (or linear orders), $\sigma:=\{\sigma_1, \sigma_2, \ldots, \sigma_r\}$ of $[n]$, is called a $k$-suitable family of permutations of $[n]$ if for all $k$-sized subsets $A$ of $[n]$ and an element $x \in A$ there exists a permutation $\sigma_i \in \sigma$ such that $x$ leads all the elements $y \in A \setminus \{x\}$ in $\sigma_i$; i.e., $y<_{\sigma_i} x$ for all $y\in A\setminus\{x\}$.
\end{definition}
The following lemma is due to Spencer~\cite{spencer1972minimal} though the exact value of $k$ and $n$ are worked out by Scott and Wood in Lemma 5 of~\cite{scott2020better}. We shall use the same values in our calculations too.  

\begin{lemma}[Spencer~\cite{spencer1972minimal}]
	\label{lem:spencers}
	For every $k \geq 2$ and $n \geq 10^4$ there is a $k$-suitable family of permutations of size at most $k 2^k \ln{\ln{n}}$.
\end{lemma}

\begin{lemma}[Lemma 12 in \cite{scott2020better}]
	\label{lem:bipartite_partitioning}
	Let $G$ be a bipartite graph with bipartition $\{A,B\}$, where vertices in $A$ have degree at most $\Delta$ and vertices in $B$ have  degree at most $d$. 
	Let $r,t,\ell$ be positive integers such that 
	\begin{equation*}
		\ell \geq e \left(\frac{ed}{r+1}\right)^{1+1/r} \text{ and } t\geq\ln( 4 d \Delta). 
	\end{equation*}
	Then there exist $t$ colorings $c_1,\dots,c_t$ of $A$, each with $\ell$ colors, such that for each vertex $v\in B$, for some coloring $c_i$, 
	each color is assigned to at most  $r$ neighbors of $v$ under $c_i$. 
\end{lemma}

We use Lemma~\ref{lem:spencers} and Lemma~\ref{lem:bipartite_partitioning} to prove the following lemma which is a prerequisite to our proof of Theorem~\ref{thm:maxdegmain}.

\begin{lemma}
\label{lem:splitGraph_partitioning}
Let $G$ be a bipartite graph with bipartition $\{A,B\}$, where vertices in $A$ have degree at most $\Delta$ and vertices in $B$ have  degree at most $d$, for some $2 \leq d \leq \Delta$.
Then,
$$\dimth(G^{*}[A,B]) \leq (81 + o(1)) d \ln{(d \Delta)} \ln\ln\Delta(2e)^{\sqrt{\ln{d}}},$$
when $d\to\infty$. 
\end{lemma}

\begin{proof}
	We follow the proof idea of Lemma 13 in~\cite{scott2020better}. 
	Let $r = \left\lceil \sqrt{\ln{d}} \right\rceil$, $\ell = \left\lceil e \left(\frac{ed}{r+1}\right)^{1+1/r} \right\rceil$, and $t = \lceil \ln( 4 d \Delta) \rceil.$ 
	Hence, we know from Lemma~\ref{lem:bipartite_partitioning} that there exist $t$ colorings $c_1, c_2, \dots,c_t$ of $A$, each with $\ell$ colors, such that for each vertex $v \in B$, for some coloring $c_j$, each color is assigned to at most $r$ neighbors of $v$ under $c_j$. 
	To obtain the threshold dimension of $G^{*}[A,B]$ 
	we further partition $B$ sequentially into $t$ parts, namely $B_1,B_2, \ldots, B_t$, based on $t$ colorings of $A$. A vertex $v \in B$ is in $B_j$ if and only if $j$ is the smallest integer such that each color appears on at most $r$ neighbors of $v$ under $c_j$. 
	For a particular coloring $c_j$ and $1 \leq k \leq \ell$, we define $A_{j,k}$ as the set containing all the vertices $v \in A$ such that $c_j(v) = k$. 
	Let $G_{j,k}$ be the  supergraph of $G^{*}[A,B]$ obtained from $G^*[A_{j,k},B_j]$ by adding all the vertices that are not present in $A_{j,k} \cup B_j$ as universal vertices. Let $H$ be the threshold supergraph of $G^{*}[A,B]$  defined as: $V(H) = A \cup B$, $E(H) = \{uv\colon u \in B,~v\in V(H)\setminus\{u\}\}$. Then we have the following:
	\begin{equation}
		G^{*}[A,B] = H \cap \left(\bigcap \limits_{1 \leq j \leq t} \bigcap \limits_{1 \leq k \leq \ell} G_{j,k}\right).
	\end{equation}
	Now we are going to calculate $\dimth(G_{j,k})$. In order to use the kind of threshold supergraphs defined in Definition~\ref{def:thsupgraph}, we need an ordering of the vertices in $A_{j,k}$, which is an independent set in $G_{j,k}$. Let $G'$ denote the graph with $V(G') = A_{j,k}$ and two vertices $x,y \in A_{j,k}$ are adjacent in $G'$ if and only if they have a common neighbor in $B_j$. We properly color $G'$ using $r\Delta + 1$ colors as the maximum degree of a vertex in $G'$ is at most $r\Delta$. Let the color classes be $C_1, C_2, \ldots , C_{r\Delta + 1}$.
	Then $A_{j,k} = C_1 \uplus C_2 \uplus \cdots \uplus C_{r\Delta + 1}$ and in $G_{j,k}$, every vertex in $B_j$ has at most one neighbor in each color class $C_i$.  We determine the ordering of the vertices in $A_{j,k}$ based on an $(r+1)$-suitable family of permutations, $\sigma_1, \sigma_2, \ldots, \sigma_p$, of $C_1, C_2, \ldots , C_{r\Delta + 1}$. From Lemma~\ref{lem:spencers}, we can assume that $p \leq (r+1) 2^{(r+1)} \ln{\ln{(r\Delta+1)}}$ . 
	From each $\sigma_a$, where $1 \leq a \leq p$, we construct two linear orderings $\sigma_a^{1}$ and $\sigma_a^{2}$ of $A_{j,k}$ as described below:
	$$\sigma_a^{1}:= \psi_{\sigma_a(1)}, \psi_{\sigma_a(2)}, \ldots , \psi_{\sigma_a(r \Delta +1)}~,$$
	$$\sigma_a^{2}:= \psi_{\sigma_a(1)}^{-1}, \psi_{\sigma_a(2)}^{-1}, \ldots , \psi_{\sigma_a(r \Delta +1)}^{-1}~.$$
	In the above, for $1\leq i\leq r\Delta+1$, $\psi_i$ denotes an arbitrary ordering of the vertices of $C_i$
	and $\psi_i^{-1}$ denotes the reverse of $\psi_i$. Now that we have total orderings $\sigma_a^{1}$ and $\sigma_a^{2}$ of $A_{j,k}$, we consider the two threshold supergraphs $\tau(G_{j,k},A_{j,k},\sigma_a^{1})$ and $\tau(G_{j,k},A_{j,k},\sigma_a^{2})$. 
	\begin{claim}
		$G_{j,k} = \bigcap \limits_{1 \leq a \leq p} (\tau(G_{j,k},A_{j,k},\sigma_a^{1}) \cap \tau(G_{j,k},A_{j,k},\sigma_a^{2})).$
	\end{claim}
	\begin{proof}
		It is clear from Definition \ref{def:thsupgraph}, we know that if $uv \in E(G_{j,k})$, then $uv$ is present in both $\tau(G_{j,k},A_{j,k},\sigma_a^{1})$ and $\tau(G_{j,k},A_{j,k},\sigma_a^{2})$, $\forall a \in [p]$. Hence we only need to show that if $uv \notin E(G_{j,k})$ then there exists at least one threshold supergraph in the collection where $u$ and $v$ are non-adjacent. If $u,v \in A_{j,k}$ then $uv \notin E(\tau(G_{j,k},A_{j,k},\sigma_a^{1}))$ and $uv \notin E(\tau(G_{j,k},A_{j,k},\sigma_a^{2}))$, $\forall a \in [p]$. Without loss of generality, assume $u \in A_{j,k}$ and $v \in B_j$. Also assume that $u$ belongs to the color class $C \in \{C_1, C_2, \ldots , C_{r\Delta + 1}\}$. We know from the property of the color classes $C_i$ that $v$ has at most one neighbor in every $C_i$ (in particular, in $C$). Suppose $|N_{G_{j,k}}(v) \cap C| = 0$. 
		We know that a vertex $v \in B_j$ has at most $r$ neighbors in $A_{j,k}$. Since we have performed $(r+1)$-suitability on the color classes $C_1, C_2, \ldots , C_{r\Delta + 1}$, there exists a permutation $\sigma \in \{\sigma_1, \sigma_2, \ldots, \sigma_p\}$ where $C$ succeeds all the color classes that contain a neighbor of $v$. Thus, $u$ succeeds all the neighbors of $v$ in $A_{j,k}$ in both $\sigma^{1}$ and $\sigma^{2}$. Hence, $u$ and $v$ are non-adjacent in both $\tau(G_{j,k},A_{j,k},\sigma^{1})$ and $\tau(G_{j,k},A_{j,k},\sigma^{2})$. Suppose $|N_{G_{j,k}}(v) \cap C| = 1$. Let $\{w\} = N_{G_{j,k}}(v) \cap C$. There exists a permutation $\sigma \in \{\sigma_1, \sigma_2, \ldots, \sigma_p\}$ such that $C$ succeeds all the other color classes that contain a neighbor of $v$ in $\sigma$. Then, $w$ succeeds all the neighbors of $v$ in $A_{j,k}$ in both $\sigma^{1}$ and $\sigma^{2}$. Since $u$ succeeds $w$ in one of $\sigma^1$ or $\sigma^2$, it follows that $u$ and $v$ are non-adjacent in either $\tau(G_{j,k},A_{j,k},\sigma^{1})$ or $\tau(G_{j,k},A_{j,k},\sigma^{2})$. 
		\hfill\qed
	\end{proof}
	Therefore, $\dimth(G_{j,k}) \leq 2p$. Now from (1) we can write:
	$$
	\dimth(G^{*}[A,B]) \leq 1 + 2p t \ell 
	$$
	Before substituting  the values of $p, \ell,$ and $t$ in the above inequality, we simplify them below.
	\begin{eqnarray*}
	p&\leq&(r+1) 2^{r+1} \ln\ln(r\Delta+1) \leq (r+1) 2^{r+1} \ln\ln\left(r\Delta\left(1+\frac{1}{r\Delta}\right)\right)\\
&\leq&(r+1) 2^{r+1} \ln\ln\left(r\Delta \cdot e^{\frac{1}{r\Delta}}\right)
 =(r+1) 2^{r+1} \ln\left(\ln{r\Delta} + \frac{1}{r\Delta}\right)\\
 &\leq&(r+1) 2^{r+1} \ln\left(\ln\Delta\left(1 + \frac{r\Delta\ln r + 1 }{r\Delta\ln\Delta}\right)\right) = (r+1) 2^{r+1} \ln{(\ln\Delta(1 + o(1)))}\\
  &=&(1+o(1))(r+1) 2^{r+1} \ln\ln\Delta\\
	t&=& \lceil \ln{(4d\Delta)} \rceil \leq \ln{4} + \ln{(d\Delta)} + 1 = \left(1+\frac{1+ \ln{4}}{\ln{(d\Delta)}}\right)\ln{(d\Delta)} \leq (1+o(1))\ln{(d\Delta)}\\
\ell&=&\lceil e\left(\frac{ed}{r+1}\right)^{1+\frac{1}{r}} \rceil \leq  e^{2 + \frac{1}{r}} \cdot \big ( \frac{d}{r+1}\big)^{1+\frac{1}{r}} + 1 \leq  e^3 \cdot \left( \frac{d}{r+1}\right)^{1+\frac{1}{r}} + 1\\
&=&(1+o(1))e^{3} \left( \frac{d}{r+1}\right)^{1+\frac{1}{r}}
\end{eqnarray*}

	\begin{align*}
		\dimth(G^{*}[A,B]) &\leq 1 + \bigg(2 \cdot (1+o(1))(r+1) 2^{r+1} \ln\ln\Delta\\
		&\hspace{.75in}\cdot (1+o(1))\ln{(d\Delta)}\cdot (1+o(1))e^{3} \left( \frac{d}{r+1}\right)^{1+\frac{1}{r}}\bigg)\\
		&\leq 1 + \left(4 e^3(1+o(1)) d \ln{(d \Delta)} \ln\ln\Delta(2^r d^{\frac{1}{r}}) \frac{1}{(r+1)^{\frac{1}{r}}}\right)\\
		&= 1 + \left((4 e^3 + o(1)) d \ln{(d \Delta)} \ln\ln\Delta(2e)^{\sqrt{\ln{d}}}\right)\\
		& \leq (81 + o(1)) d \ln{(d \Delta)} \ln\ln\Delta(2e)^{\sqrt{\ln{d}}}.
	\end{align*}
\hfill\qed
\end{proof}

\subsection{Proof of the main theorem}
We need the following partitioning lemma by Scott and Wood.
\begin{corollary}[Corollary 11 in \cite{scott2020better}]
	\label{cor:partitioning_lemma}
	For every graph $G$ with maximum degree $\Delta\geq 2$ and for all integers $d\geq 100 \ln\Delta$ and  $k\geq \frac{3\Delta}{d}$,  there is a partition $V_1,\dots,V_k$ of $V(G)$, such that $|N_G(v)\cap V_i| \leq d$ for each $v\in V(G)$ and $i\in[k]$. 
\end{corollary}

\begin{theorem}
\label{thm:maxdegmain}
For a graph $G$ with maximum degree $\Delta$,
\begin{equation*}
\dimth(G) \leq (24300 + o(1)) \Delta \ln^{2}{\Delta} \ln\ln\Delta(2e)^{\sqrt{(1+o(1))\ln{\ln{\Delta}}}},
\end{equation*}
when $\Delta \rightarrow \infty$.
\end{theorem}

\begin{proof}
	Let $d = \lceil 100 \ln{\Delta} \rceil$ and $k = \lceil \frac{3\Delta}{d} \rceil$. Using Corollary \ref{cor:partitioning_lemma}, we get a partition of $V(G)$ into $k$ parts, $V_1,V_2,\ldots,V_k$, such that for any vertex $v \in V(G)$, $|N_G(v) \cap V_i| \leq d$, where $1 \leq i \leq k$. Since the maximum degree of $G[V_i]$ is $d$, we can do a proper coloring of $G[V_i]$ using $d+1$ colors. Therefore, $\forall i \in [k]$ each part $V_i$ can further be partitioned into $d+1$ parts, namely $V_i^1, V_i^2, \ldots, V_i^{d+1}$, where each part is an independent set in $G$. 
	\begin{claim}
		\begin{equation*}
			G= \bigcap \limits_{1 \leq i \leq k} \bigcap \limits_{1 \leq j \leq d+1} G^{*}[V_i^j, V(G) \setminus V_i^j].
		\end{equation*}
	\end{claim}
	\begin{proof}
		From the fact that $V_i^j$ is an independent set in $G$ and from the construction of $G^{*}[V_i^j, V(G) \setminus V_i^j]$, it is clear that $G^{*}[V_i^j, V(G) \setminus V_i^j]$, for $i \in [k], j \in [d+1]$, is a supergraph of $G$. Suppose that $uv \notin E(G)$.
		If $u, v \in V_i^j$ for some $i \in [k]$ and $j\in [d+1]$, then $u$ and $v$ are non-adjacent in $G^{*}[V_i^j, V(G) \setminus V_i^j]$. Otherwise, $u\in V_i^j$ for some $i\in [k]$ and $j\in [d+1]$, and $v\in V(G)\setminus V_i^j$, in which case $u$ and $v$ are non-adjacent in $G^{*}[V_i^j, V(G) \setminus V_i^j]$.
		\hfill\qed
	\end{proof}
	Applying Lemma~\ref{lem:splitGraph_partitioning} we can write, 
	\begin{align*}
		\dimth(G) 
		& \leq k \cdot (1+o(1))d \cdot (81 + o(1)) d \ln{(d \Delta)} \ln\ln\Delta(2e)^{\sqrt{\ln{d}}} \\
		& \leq (243 + o(1)) \Delta d \ln{\Delta} \ln\ln\Delta(2e)^{\sqrt{\ln{d}}}\qquad\left(\mbox{since }k = \left\lceil \frac{3\Delta}{d}\right\rceil \right)\\
		& \leq (24300 + o(1)) \Delta \ln^{2}{\Delta} \ln\ln\Delta(2e)^{\sqrt{(1+o(1))\ln{\ln{\Delta}}}}\\
		&\hspace{2.5in}(\mbox{since }d = \left\lceil 100 \ln \Delta \right\rceil )  
	\end{align*}
\hfill\qed
\end{proof}

Since $(2e)^{\sqrt{(1+o(1))\ln{\ln{\Delta}}}} \ln\ln\Delta = (\ln{\Delta})^{\frac{\ln(2e)\sqrt{(1+o(1))\ln\ln\Delta} }{\ln\ln\Delta}+ \frac{\ln{\ln\ln\Delta}}{\ln\ln\Delta}} = \ln^{o(1)}{\Delta}$ we get the following corollary.
\begin{corollary}
$$\dimth(\Delta) \in O(\Delta \ln^{2+o(1)}{\Delta}).$$
\end{corollary}
\section{Threshold dimension and degeneracy}
\label{sec_degeneracy}
Given a graph $G$ and a positive integer $k$, an ordering of the vertices of $G$ such that no vertex has more than $k$ neighbors after it is called a \emph{$k$-degenerate ordering} of $G$.  We say a graph is \emph{$k$-degenerate} if it has a $k$-degenerate ordering. The minimum $k$ such that $G$ is $k$-degenerate is called the \emph{degeneracy} of $G$. From its definition, it is clear that the degeneracy of a graph is at most its  maximum degree. 
In this section, we derive upper bounds on the threshold dimension of a graph in terms of its degeneracy. The techniques we adopt are mostly inspired by those in~\cite{AdiSunRog}.

Throughout this section, we shall assume that $G$ is a $k$-degenerate graph on $n$ vertices with vertex set $\{v_1, v_2,\ldots , v_n\}$ and that $v_1, v_2, \ldots , v_n$ is a $k$-degenerate ordering of $G$. Thus, for each $i\in\{1,2,\ldots,n\}$,  $|N_G(v_i)\cap\{v_{i+1},v_{i+2},\ldots,v_n\}| \leq k$. The vertices in $N_G(v_i)\cap\{v_{i+1},v_{i+2},\ldots,v_n\}$ are called the \emph{forward neighbors} of $v_i$. Let  $i < j$ and $v_iv_j \notin E(G)$. A coloring $f$ of the vertices of $G$ is \emph{desirable} for  the non-adjacent pair $(v_i,v_j)$ if (i) $f$ is a proper coloring, and (ii) $f(v_j) \neq f(v_t)$, for all neighbors $v_t$ of $v_i$ such that $t>j$.  

\begin{lemma}
\label{lem_desirable_coloring}
Let $G$ be a $k$-degenerate graph on $n$ vertices and let $v_1, v_2, \ldots , v_n$ be a $k$-degenerate ordering of $G$. Let $r = \lceil \ln n\rceil$. Then there is a collection $\{f_1, \ldots , f_r\}$, where each $f_i\colon V(G)\rightarrow [10k]$ is a proper coloring of the vertices of $G$, such that for every non-adjacent pair $(v_i,v_j)$, where $i<j$, there exists an $\ell \in [r]$ such that $f_{\ell}$ is a desirable coloring for the pair $(v_i,v_j)$.  
\end{lemma}

\begin{proof}
	We explain the randomized procedure for constructing the coloring $f_1$ below. Start coloring the vertices from $v_n$ and color them all the way down to $v_1$ in the following way. Assume we have colored the vertices $v_{n}$ to $v_{i+1}$ and are about to color $v_i$. From the set of $10k$ colors, remove the colors that have been assigned to the forward neighbors of $v_i$. This leaves us with a set of at least $9k$ colors. Uniformly at random, choose one color from this set and assign it to $v_i$. This completes our description of the construction of the coloring $f_1$. The procedure ensures that $f_1$ is a proper coloring. Independently, repeat the above procedure to construct the colorings $f_2,f_3, \ldots , f_r$. 
	
	Consider a non-adjacent pair $(v_i, v_j)$, where $i<j$. The probability that $f_1$ is not a desirable coloring for this pair is equal to the probability that  a forward neighbor of $v_i$ that is after $v_j$ in the $k$-degenerate ordering gets the same color as that of $v_j$. This probability is at
	most $k/9k = 1/9$.  Let $A_{i,j}$ denote the bad event that none of the colorings $f_1, f_2,\ldots , f_r$ is a desirable
	coloring for the pair $(v_i,v_j)$. Then, $Pr[A_{i,j}] \leq 1/9^r < 1/n^2$. Applying the union bound, $Pr[\bigcup
	_{v_iv_j \notin E(G),~i<j} A_{i,j}] \leq \sum\limits_{v_iv_j \notin E(G),~i<j}Pr[A_{i,j}] < {n \choose 2}\frac{1}{n^2} < 1$. Thus, the statement of the lemma holds with non-zero probability.
	\hfill\qed
\end{proof}

\begin{theorem}
\label{thm_thresh_dim_degeneracy}
Let $G$ be a $k$-degenerate graph on $n$ vertices. Then, $\dimth(G) \leq 10k\ln n$. 
\end{theorem}

\begin{proof}
	Let $V(G) = \{v_1, \ldots , v_n\}$ and let $\sigma\colon  v_1, v_2, \ldots, v_n$ be a $k$-degenerate ordering of $G$. Let $\{f_1, \ldots , f_{\lceil \ln n \rceil}\}$ be the collection of proper colorings of $V(G)$, where each coloring uses at most $10k$ colors, given by Lemma~\ref{lem_desirable_coloring}. For each coloring $f_a$, $a \in [\lceil \ln n \rceil]$, and each color $b \in [10k]$, we construct a threshold supergraph $T_{a,b}$ of $G$ as follows. Let $C_b^a = \{v \in V(G)\colon f_a(v) = b\}$. Since $f_a$ is a proper coloring, $C_b^a$ is an independent set. We define $T_{a,b} := \tau(G,C_b^a, \sigma|_{C_b^a})$ (see Definition \ref{def:thsupgraph} and Proposition \ref{prop:thsupgraph}). 
	
	We claim that $G = \bigcap_{a \in [\lceil \ln n \rceil],~b\in [10k]} T_{a,b}$. Since each $T_{a,b}$ is a supergraph, all we need to do is to show that for every non-adjacent pair $(v_i, v_j)$ in $G$, where $i<j$, there is a threshold supergraph in our collection that does not contain the edge $v_iv_j$. Assume $f_a$ is a desirable coloring for $(v_i, v_j)$ and $f_a(v_j) = b$ (Lemma~\ref{lem_desirable_coloring} guarantees that such a coloring exists). Then, we claim that $v_iv_j  \notin E(T_{a,b})$. If $f_a(v_i) = b$, then $v_iv_j  \notin E(T_{a,b})$ as $C_b^a$ is an independent set in $T_{a,b}$. Suppose $f_a(v_i) \neq b$. Since no neighbor $u$ of $v_i$ that is after $v_j$ in the $k$-degenerate ordering has $f_a(u) = b$, all the neighbors of $v_i$ in $C_b^a$ appear before $v_j$ in the ordering $\sigma|_{C_b^a}$. Thus, $v_iv_j  \notin E(T_{a,b})$  . This completes the proof of the theorem. 
	\hfill\qed
\end{proof}

\subsection{Random graphs}
The following lemma was proved in \cite{AdiSunRog}. 
\begin{lemma}[Lemma 12 in \cite{AdiSunRog}]
\label{lem_random_graph_GNP}
For a random graph $G \in \mathcal{G}(n,p)$, where $p = \frac{c}{n-1}$ and $1 \leq c \leq n-1$, $Pr[G \mbox{ is } 4ec
 \mbox{-degenerate}] \geq 1 - \frac{1}{\Omega(n^2)}$. 
\end{lemma}
Applying Lemma  \ref{lem_random_graph_GNP} and Theorem \ref{thm_thresh_dim_degeneracy}, we get the following lemma.
\begin{lemma}
\label{lem_thdim_random_graph_GNP}
For a random graph $G \in \mathcal{G}(n,p)$, where $p = \frac{c}{n-1}$ and $1 \leq c \leq n-1$, $Pr[\dimth(G) \in O(c\ln n)] \geq 1 - \frac{1}{\Omega(n^2)}$. 
\end{lemma}
It is known that (see page 35 of \cite{bollobas})
\begin{eqnarray}
P_m(Q) & \leq & 3\sqrt{m}P_p(Q) \label{Eqn_GNP_GNM} 
\end{eqnarray}
where (i) $Q$ is a property of graphs of order $n$, (ii) $P_m(Q)$ is the probability that Property Q is satisfied by a graph $G \in \mathcal{G}(n,m)$, and (iii) $P_p(Q)$ is the probability that Property Q is satisfied by a graph $G \in \mathcal{G}(n,p)$ with $p = \frac{m}{{n \choose 2}} = \frac{2m/n}{n-1}$. Assume $m \geq n/2$. Then, $p = \frac{2m/n}{n-1} \geq \frac{1}{n-1}$ and by Lemma \ref{lem_thdim_random_graph_GNP}, $Pr[\dimth(G) \notin O(\frac{2m}{n}\ln n)] \leq \frac{1}{\Omega(n^2)}$. Applying Equation \ref{Eqn_GNP_GNM}, for a random graph $G \in \mathcal{G}(n,m)$,  $m \geq n/2$, $Pr[\dimth(G) \notin O(\frac{2m}{n}\ln n)] \leq \frac{3\sqrt{m}}{\Omega(n^2)} \leq \frac{1}{\Omega(n)}$. We thus have the following theorem.  
\begin{theorem}
\label{thm_random_graph_GNM}
For a random graph $G \in \mathcal{G}(n,m)$, $m \geq n/2$, $Pr[\dimth(G) \in O(\frac{2m}{n}\ln n)] \geq 1 - \frac{1}{\Omega(n)}$. In other words, $Pr[\dimth(G) \in O(d_{av}\ln n)] \geq 1 - \frac{1}{\Omega(n)}$, where $d_{av}$ denotes the average degree of $G$. 
\end{theorem}
\subsection{Graphs of high girth}
The \emph{girth} of a graph is the length of a smallest cycle in it. We assume that if the graph is acyclic, then its girth is $\infty$. We apply Theorem \ref{thm_thresh_dim_degeneracy} to prove an upper bound for the threshold dimension of a graph in terms of its girth and the number of vertices. The following lemma was proved in \cite{majumder2021local}.
\begin{lemma}[Lemma 23 in \cite{majumder2021local}]
\label{lem_girth_degeneracy}
Let $G$ be a graph on $n$ vertices having girth greater than $g+1$. Then, $G$ is $k$-degenerate, where $k= \lceil n^{\frac{1}{\lfloor g/2 \rfloor }}\rceil$.
\end{lemma}
Applying the above lemma, we get the following corollary to Theorem \ref{thm_thresh_dim_degeneracy}.
\begin{corollary}
\label{cor_degeneracy_girth}
Let $G$ be a graph on $n$ vertices with girth greater than $g+1$. Then, $\dimth(G) \leq 10\lceil n^{\frac{1}{\lfloor g/2 \rfloor }}\rceil \ln n$. 
\end{corollary}
The bipartite graph $G$ obtained by removing a perfect matching from the complete bipartite graph $K_{n,n}$ is known to have a boxicity of $\frac{n}{2}$. From Observation~\ref{obv_thdim_box} and by applying Corollary~\ref{cor_degeneracy_girth} with $g=2$, we have $\frac{n}{2} \leq \dimth(G) = O(n\ln n)$. Thus, we cannot expect to get an upper bound of $O(n^{\alpha/g})$,  with $\alpha < 2$, for the threshold dimension of a graph with girth greater than $g+1$. 

\section{Threshold dimension and minimum vertex cover}
\label{sec:thdim_MVC}
A \emph{vertex cover} of $G$ is a set of vertices $S \subseteq V(G)$ such  that $\forall e \in E(G)$, at least one endpoint of $e$ is in $S$. A \emph{minimum vertex cover} of $G$ is a vertex cover of $G$ of the smallest cardinality. We use $\beta(G)$ to denote the cardinality of a minimum vertex cover. In this section, we prove a tight upper bound for the threshold dimension of a graph in terms of the size of its minimum vertex cover. 

\begin{proposition}
\label{prop:th_dim_MVC}
For a graph $G$, $\dimth(G) \leq \beta(G)$.
\end{proposition}

\begin{proof}
	Let $B$ denote a minimum vertex cover of $G$, and $b := |B| = \beta(G)$. Then, $A := V(G) \setminus B$ is a maximum independent set in $G$. Let $B = \{v_1, v_2, \ldots , v_b\}$. For each $i \in [b-1]$, we construct threshold supergraph $G_i:= \tau(G,\{v_i\},\sigma_i)$, where $\sigma_i$ denotes the trivial ordering of the vertex inside the singleton set $\{v_i\}$. To construct the last threshold supergraph $G_b$, let $\pi_b$ be an ordering of the  vertices of $A$ where every vertex in $N_G(v_b) \cap A$ appear before every vertex in $A \setminus N_G(v_{b})$. We define $G_{b}:=\tau(G,A,\pi_b)$. We claim that $G = \bigcap_{i=1}^b G_i$. We know from our construction that every $G_i$ is a supergraph of $G$. Suppose $xy \notin E(G)$, for some $x,y \in V(G)$. If $x,y \in A$, then $xy \notin E(G_b)$. Assume at least one of $x$ or $y$ belongs to $B$. If $x = v_i$ or $y = v_i$, for some $i<b$, then $xy \notin E(G_i)$. We are left with the case when $x = v_b$ and $y \in A$ (or vice versa). In this case, it can be verified that $xy \notin E(G_b)$. 
	\hfill\qed
\end{proof}

Since $\alpha(G) = |V(G)| - \beta(G)$, by combining Corollary \ref{thm_follows_from_thr_cov}(a) with Proposition \ref{prop:th_dim_MVC}, we get the following theorem. 
\begin{theorem}
\label{thm:thrdim_indep_clique}
For a graph $G$ on $n$ vertices, $\dimth(G) \leq  n-\max\{\omega(G),\alpha(G)\}$. 
\end{theorem}
In Ramsey theory, $R(k,k)$ denotes the smallest positive integer $n$ such that every graph on $n$ vertices has either an independent set of size $k$ or a clique of size $k$. It is known due to \cite{conlon2009new} that $R(k,k) \leq k^{\frac{-c\ln k}{\ln\ln k}} 4^k$, where $c$ is a constant. This implies that for sufficiently large $n$, every graph on $n$ vertices has either an independent set or a clique (or both) of size $0.72\ln n$. This gives us the following corollary. 
\begin{corollary}    
\label{cor:uppboundn}
When $n$ is sufficiently large, a graph $G$ on $n$ vertices satisfies $\dimth(G) \leq n - 0.72\ln n$.
\end{corollary}
\subsubsection{Tightness of the bound in Theorem \ref{thm:thrdim_indep_clique}}
It can be verified that the graph $H$ on $2n$ vertices having threshold dimension $n$ constructed in Example \ref{example:treewidth_MVC} satisfies $\alpha(H) = \omega(H) = \beta(H)  = n$. Hence, the bounds in Theorem \ref{thm:thrdim_indep_clique} and Proposition~\ref{prop:th_dim_MVC} are tight.  
\section*{Acknowledgment} We thank Karteek Sreenivasaiah for helpful discussions and the anonymous reviewers for their valuable suggestions.
\bibliographystyle{plain}

\begin{thebibliography}{10}

\bibitem{DiptAdigaHardness}
Abhijin Adiga, Diptendu Bhowmick, and L.~Sunil Chandran.
\newblock The hardness of approximating the boxicity, cubicity and threshold
  dimension of a graph.
\newblock {\em Discrete Applied Mathematics}, 158(16):1719--1726, 2010.

\bibitem{AdiSunRog}
Abhijin Adiga, L.~Sunil Chandran, and Rogers Mathew.
\newblock Cubicity, degeneracy, and crossing number.
\newblock {\em European Journal of Combinatorics}, 35:2--12, 2014.

\bibitem{bollobas}
B{\'e}la Bollob{\'a}s.
\newblock {\em Random graphs}, volume~73.
\newblock Cambridge University Press, 2001.

\bibitem{chacko2020representing}
Daphna Chacko and Mathew~C. Francis.
\newblock Representing graphs as the intersection of cographs and threshold
  graphs.
\newblock {\em Electronic Journal of Combinatorics}, 28(3):P3.11, 2021.

\bibitem{tech-rep}
L.~Sunil Chandran, Mathew~C. Francis, and Naveen Sivadasan.
\newblock Geometric representation of graphs in low dimension using axis
  parallel boxes.
\newblock {\em Algorithmica}, 56(2):129--140, 2010.

\bibitem{CHANDRAN2007733}
L.~Sunil Chandran and Naveen Sivadasan.
\newblock Boxicity and treewidth.
\newblock {\em Journal of Combinatorial Theory, Series B}, 97(5):733--744,
  2007.

\bibitem{chvtal1977aggregation}
V{\'a}clav Chv{\'a}tal and Peter~L. Hammer.
\newblock Aggregation of inequalities in integer programming.
\newblock {\em Annals of Discrete Mathematics}, 1:145--162, 1977.

\bibitem{conlon2009new}
David Conlon.
\newblock A new upper bound for diagonal ramsey numbers.
\newblock {\em Annals of Mathematics}, pages 941--960, 2009.

\bibitem{dushnik1950concerning}
Ben Dushnik.
\newblock Concerning a certain set of arrangements.
\newblock {\em Proceedings of the American Mathematical Society},
  1(6):788--796, 1950.

\bibitem{erdos1991dimension}
Paul Erd\H{o}s, Henry~A. Kierstead, and William~T. Trotter.
\newblock The dimension of random ordered sets.
\newblock {\em Random Structures \& Algorithms}, 2(3):253--275, 1991.

\bibitem{boxMinor}
Louis Esperet and Veit Wiechert.
\newblock Boxicity, poset dimension, and excluded minors.
\newblock {\em Electronic Journal of Combinatorics}, 25(4):P4.51, 2018.

\bibitem{golumbic2004algorithmic}
Martin~C. Golumbic.
\newblock {\em Algorithmic graph theory and perfect graphs}.
\newblock Elsevier, 2004.

\bibitem{hammermahadev}
Peter~L. Hammer and Nadimpalli~V.R. Mahadev.
\newblock Bithreshold graphs.
\newblock {\em SIAM Journal on Algebraic Discrete Methods}, 6(3):497--506,
  1985.

\bibitem{KORACH199397}
Ephraim Korach and Nir Solel.
\newblock Tree-width, path-width, and cutwidth.
\newblock {\em Discrete Applied Mathematics}, 43(1):97--101, 1993.

\bibitem{Kratochvil}
Jan Kratochv\'il.
\newblock A special planar satisfiability problem and a consequence of its
  {NP}--completeness.
\newblock {\em Discrete Applied Mathematics}, 52:233--252, 1994.

\bibitem{kratotuza}
Jan Kratochv\'il and Zsolt Tuza.
\newblock Intersection dimensions of graph classes.
\newblock {\em Graphs and Combinatorics}, 10(2-4):159--168, 1994.

\bibitem{mahadev1995threshold}
Nadimpalli~V.R. Mahadev and Uri~N. Peled.
\newblock {\em Threshold graphs and related topics}.
\newblock Elsevier, 1995.

\bibitem{mahajan}
Meena Mahajan.
\newblock Depth-2 threshold circuits.
\newblock {\em Resonance}, 24(3):371--380, 2019.

\bibitem{majumder2021local}
Atrayee Majumder and Rogers Mathew.
\newblock Local boxicity and maximum degree, Available at
  https://arxiv.org/abs/1810.02963, 2021.

\bibitem{raschle1995recognition}
Thomas Raschle and Klaus Simon.
\newblock Recognition of graphs with threshold dimension two.
\newblock In {\em Proceedings of the twenty-seventh annual ACM symposium on
  Theory of computing}, pages 650--661, 1995.

\bibitem{Roberts}
Fred~S. Roberts.
\newblock {\em Recent Progresses in Combinatorics}, chapter On the boxicity and
  cubicity of a graph, pages 301--310.
\newblock Academic Press, New York, 1969.

\bibitem{ROBERTSON1986309}
Neil Robertson and Paul~D. Seymour.
\newblock Graph minors. {II}. algorithmic aspects of tree-width.
\newblock {\em Journal of Algorithms}, 7(3):309--322, 1986.

\bibitem{scott2020better}
Alex Scott and David Wood.
\newblock Better bounds for poset dimension and boxicity.
\newblock {\em Transactions of the American Mathematical Society},
  373(3):2157--2172, 2020.

\bibitem{spencer1972minimal}
Joel Spencer.
\newblock Minimal scrambling sets of simple orders.
\newblock {\em Acta Mathematica Hungarica}, 22(3-4):349--353, 1972.

\bibitem{yannakakis1982complexity}
Mihalis Yannakakis.
\newblock The complexity of the partial order dimension problem.
\newblock {\em SIAM Journal on Algebraic Discrete Methods}, 3(3):351--358,
  1982.

\end{thebibliography}

\end{document}